\documentclass[10pt]{amsart}%
\usepackage{amssymb,amsmath,amsfonts,latexsym,amsthm,geometry}
\usepackage{amsmath}
\usepackage{amsfonts}
\usepackage{amssymb}
\usepackage{graphicx}
\usepackage[colorlinks=true,linkcolor=blue,citecolor=blue]{hyperref}%
\providecommand{\U}[1]{\protect\rule{.1in}{.1in}}
\providecommand{\U}[1]{\protect\rule{.1in}{.1in}}
\providecommand{\U}[1]{\protect\rule{.1in}{.1in}}
\providecommand{\U}[1]{\protect\rule{.1in}{.1in}}
\providecommand{\U}[1]{\protect\rule{.1in}{.1in}}
\newtheorem{theorem}{Theorem}[section]

\newtheorem{proposition}[theorem]{Proposition}
\newtheorem{lemma}[theorem]{Lemma}

\theoremstyle{definition}

\begin{document}
\title[The best constants in the Multiple Khintchine Inequality]{The best constants in the Multiple Khintchine Inequality}
\author[D. N\'u\~nez]{Daniel N\'{u}\~{n}ez-Alarc\'{o}n}
\address{Departamento de Matem\'{a}ticas\\
\indent Universidad Nacional de Colombia\\
\indent111321 - Bogot\'a, Colombia}
\email{danielnunezal@gmail.com}
\author[D. M. Serrano]{Diana Marcela Serrano-Rodr\'{\i}guez}
\address{Departamento de Matem\'{a}ticas\\
\indent Universidad Nacional de Colombia\\
\indent111321 - Bogot\'{a}, Colombia}
\email{dmserrano0@gmail.com}
\thanks{}

\begin{abstract}
In this work we provide the best constants of the multiple Khintchine
inequality. This allows us, among other results, to obtain the best constants
of the mixed $\left(  \ell_{\frac{p}{p-1}},\ell_{2}\right)  $-Littlewood
inequality, thus ending completely a work started by Pellegrino in \cite{pell}.

\end{abstract}
\maketitle

\section{Introduction}

Let $\left(  a_{i}\right)  _{i=1}^{\infty}$ be a scalar sequence. In 1922 Hans
Rademacher showed that if, for almost all choice of signs $\delta_{i}$, the
series $\sum_{i=1}^{\infty}a_{i}\delta_{i}$ converges, $\sum_{i=1}^{\infty
}\left\vert a_{i}\right\vert ^{2}$ also does it. A few months later, Aleksandr
Khintchine showed the following more general result (see \cite{khinchine}):

\begin{theorem}
[Khintchine inequality]For any $0<p<\infty$, there is a positive constant
$\mathrm{A}_{p}$ (depends only on $p$) such that regardless of the scalar
sequence $(a_{j})_{j=1}^{n}\in\mathbb{K}^{n}$,we have
\begin{equation}
\left(  \sum_{j=1}^{n}|a_{j}|^{2}\right)  ^{\frac{1}{2}}\leq\mathrm{A}%
_{p}\left(  \int_{0}^{1}\left\vert \sum_{j=1}^{n}a_{j}r_{j}(t)\right\vert
^{p}dt\right)  ^{\frac{1}{p}}, \label{khin}%
\end{equation}
where $r_{j}$ denotes the $j-$th Rademacher function.
\end{theorem}

Above and henceforth, $\mathbb{K}$ denotes the scalar field $\mathbb{R}$ or
$\mathbb{C}$.\newline

The \textit{Rademacher functions} are defined as follows
\begin{align*}
r_{j}  &  :\left[  0,1\right]  \longrightarrow\mathbb{R}\text{, }%
j\in\mathbb{N}\\
r_{j}\left(  t\right)   &  :=sign\left(  \sin2^{j}\pi t\right)
\end{align*}

In probabilistic terms, the system $\left(  r_{j}\left(  t\right)  \right)
_{j=1}^{\infty}$\ is a sequence of independent identically distributed random
variables defined on the closed interval $\left[  0,1\right]  $ with a
Lebesgue measure on its Borel subsets. Up to a linear transformation,
Rademacher functions describe binomial trials with probability of success
equal to $\frac{1}{2}$.\newline

These functions have, among others, the following important properties: The
\textit{Rademacher functions} are an orthonormal sequence in $L_{2}\left[
0,1\right]  $, and then
\begin{equation}
\int_{0}^{1}\left\vert \sum_{j=1}^{\infty}a_{j}r_{j}\left(  t\right)
\right\vert ^{2}dt=\sum_{j=1}^{\infty}\left\vert a_{j}\right\vert ^{2}
\label{jo19}%
\end{equation}
for all $a=\left(  a_{j}\right)  _{j=1}^{\infty}\in\ell_{2}$.

On the other hand, for every $n\in\mathbb{N}$ and $\delta_{1},\delta
_{2},...,\delta_{n}=\pm1$, $\sum_{j=1}^{n}a_{j}r_{j}\left(  t\right)  $ and
$\sum_{j=1}^{n}\delta_{j}a_{j}r_{j}\left(  t\right)  $ have the same
distribution in $L_{p}$, for all $p>0$. That means, for every $n\in\mathbb{N}$
and $\delta_{1},\delta_{2},...,\delta_{n}=\pm1$,%
\begin{equation}
\int_{0}^{1}\left\vert \sum_{j=1}^{n}a_{j}r_{j}\left(  t\right)  \right\vert
^{p}dt=\int_{0}^{1}\left\vert \sum_{j=1}^{n}\delta_{j}a_{j}r_{j}\left(
t\right)  \right\vert ^{p}dt, \label{equal}%
\end{equation}
for all $p>0$.

\smallskip

Inductively, for every $m\in\mathbb{N}$, if $n\in\mathbb{N}$ and $\left(
y_{i_{1}...i_{m}}\right)  _{i_{1,}...,i_{m}=1}^{n}$ is an array of scalars,
then for any choice of signs $\left(  \varepsilon_{i_{1}}\right)  _{i_{1}%
=1}^{n},\left(  \varepsilon_{i_{2}}\right)  _{i_{2}=1}^{n},...,\left(
\varepsilon_{i_{m}}\right)  _{i_{m}=1}^{n}\in\{-1,1\}^{n}$ we have%
\begin{align}
&  \int_{I^{m}}\left\vert \sum_{i_{1},...,i_{m}=1}^{n}y_{i_{1}...i_{m}%
}r_{i_{1}}(t_{1})\cdots r_{i_{m}}(t_{m})\right\vert ^{p}dt_{1}\cdots
dt_{m}\label{equalm}\\
&  =\int_{I^{m}}\left\vert \sum_{i_{1},...,i_{m}=1}^{n}y_{i_{1}...i_{m}%
}r_{i_{1}}(t_{1})\cdots r_{i_{m}}(t_{m})\varepsilon_{i_{1}}\cdots
\varepsilon_{i_{m}}\right\vert ^{p}dt_{1}\cdots dt_{m}\nonumber
\end{align}
for all $p>0$. Above and henceforth $I$ denotes the closed interval $[0,1]$.

Using duality and (\ref{jo19}) we also get a similar upper bound in
(\ref{khin}). In other words, Khintchine inequality shows that we can control
the sum $\sum_{j=1}^{\infty}a_{j}r_{j}$ in any $L_{p}$ norm by the $\ell_{2}%
-$norm of the scalar sequence $\left(  a_{j}\right)  _{j=1}^{\infty}$.

\smallskip

Khintchine inequality is originally a result from probability, but it is also
frequently used in Analysis and Topology (see \cite{bob, wast, Di, Eich,
garling, gine, pel}). The importance of the Rademacher functions and the
Khintchine inequality in Functional Analysis lies mainly on the fact of its
utility in the study of the geometry of Banach spaces (see \cite{bob, Di,
garling}). Furthermore, the concern of the Rademacher functions in the theory
of functional and trigonometric series and in the theory of Banach spaces is
well known and it is commonly attributable to stochastic independence of the
Rademacher functions. One of the main manifestations of this stochastic
independence is, namely, the Khintchine inequality. Moreover, the Khintchine
inequality (and also related results and its variants) is an important
auxiliary result frequently used to prove results concerning to summability,
specially the case in which $1\leq p\leq2$ (see \cite{Di}).

\smallskip

The optimal constants $\mathrm{A}_{p}$ are known; obviously $\mathrm{A}%
_{p}=1,$ for $2\leq p<\infty,$ and by the other hand, Uffe Haagerup
(\cite{Ha}) proved that
\[
\mathrm{A}_{p}=\frac{1}{\sqrt{2}}\left(  \frac{\Gamma\left(  \frac{p+1}%
{2}\right)  }{\sqrt{\pi}}\right)  ^{-\frac{1}{p}},\ \ \text{ for }1.85\approx
p_{0}<p<2
\]
and
\[
\mathrm{A}_{p}=2^{\frac{1}{p}-\frac{1}{2}},\ \ \text{ for }0<p\leq
p_{0}\approx1.85.
\]
The exact definition of $p_{0}$ is the following: $p_{0}\in(1,2)$ is the
unique real number satisfying%

\begin{equation}
\Gamma\left(  \frac{p_{0}+1}{2}\right)  =\frac{\sqrt{\pi}}{2}. \label{pezero}%
\end{equation}
{}

In this work we will show, among other results, an interesting connection
between the Khintchine inequality and a famous inequality of Hardy and Littlewood.

\smallskip

For $p\geq1$, we introduce the following notation: $X_{p}:=\ell_{p}$ and
$X_{\infty}:=c_{0}$. From now on, $\left(  e_{k}\right)  _{k=1}^{\infty}$
denotes the sequence of canonical vectors in $X_{p}$. For any function $f$ we
shall consider $f(\infty):=\lim_{s\rightarrow\infty}f(s)$ and for any $s>1$ we
denote the conjugate index of $s$ by $s^{\ast},$ i.e., $\frac{1}{s}+\frac
{1}{s^{\ast}}=1$. Furthermore, we denote $1^{\ast}$ by $\infty$.

\smallskip Let $p,q\in\lbrack1,\infty]$. We recall that for a continuous
bilinear form $T:X_{p}\times X_{q}\rightarrow\mathbb{K}$:%
\begin{align*}
\left\Vert T\right\Vert  &  =\sup\left\{  \left\vert T\left(  x,y\right)
\right\vert :\left\Vert x\right\Vert _{X_{p}}\leq1,~\left\Vert y\right\Vert
_{X_{q}}\leq1\right\} \\
&  =\sup\left\{  \left\vert \sum_{i,j}a_{ij}x_{i}y_{j}\right\vert :\left\Vert
x\right\Vert _{X_{p}}\leq1,~\left\Vert y\right\Vert _{X_{q}}\leq1\right\}  ,
\end{align*}
where $T\left(  e_{i},e_{j}\right)  =a_{ij}$, for all $i,j\in\mathbb{N}.$

The Hardy--Littlewood inequality (\cite{hardy}, 1934) is a continuation of a
famous work of Littlewood (\cite{LLL}, 1930) and can be stated as follows:

\begin{itemize}
\item \cite[Theorems 2 and 4]{hardy} If $p,q\in\lbrack2,\infty]$ are such
that
\[
\frac{1}{2}<\frac{1}{p}+\frac{1}{q}<1
\]
then there is a constant $C_{p,q}\geq1$ such that
\[
\left(  \sum\limits_{j,k=1}^{\infty}\left\vert A(e_{j},e_{k})\right\vert
^{\frac{pq}{pq-p-q}}\right)  ^{\frac{pq-q-p}{pq}}\leq C_{p,q}\left\Vert
A\right\Vert
\]
for all continuous bilinear forms $A:X_{p}\times X_{q}\rightarrow\mathbb{K}$.
Moreover the exponent $\frac{pq}{pq-p-q}$ is optimal.

\item \cite[Theorems 1 and 4]{hardy} If $p,q\in\lbrack2,\infty]$ are such
that
\[
\frac{1}{p}+\frac{1}{q}\leq\frac{1}{2}%
\]
then there is a constant $C_{p,q}\geq1$ such that
\[
\left(  \sum\limits_{j,k=1}^{\infty}\left\vert A(e_{j},e_{k})\right\vert
^{\frac{4pq}{3pq-2p-2q}}\right)  ^{\frac{3pq-2p-2q}{4pq}}\leq C_{p,q}%
\left\Vert A\right\Vert
\]
for all continuous bilinear forms $A:X_{p}\times X_{q}\rightarrow\mathbb{K}$.
Moreover the exponent $\frac{4pq}{3pq-2p-2q}$ is optimal, where, for the case
that $p$ and $q$ are simultaneously $\infty$, the optimal exponent is $4/3$.
\end{itemize}

As mentioned in \cite{was}, an unified version of the above two results of
Hardy and Littlewood asserts that:

\begin{theorem}
\label{tony}Let $p,q\in\lbrack2,\infty]$ be such that $\lambda:=\frac
{pq}{pq-p-q}>0.$ There is a constant $C_{p,q}\geq1$ such that
\[
\left(  \sum\limits_{j=1}^{\infty}\left(  \sum\limits_{k=1}^{\infty}\left\vert
A(e_{j},e_{k})\right\vert ^{2}\right)  ^{\frac{\lambda}{2}}\right)  ^{\frac
{1}{\lambda}}\leq C_{p,q}\left\Vert A\right\Vert
\]
for all continuous bilinear forms $A:X_{p}\times X_{q}\rightarrow$
$\mathbb{K}$.\bigskip\ Moreover the exponents $2$ and $\lambda$ are optimal,
where, for the case that $p$ and $q$ are simultaneously $\infty$, the optimal
exponent $\lambda$ is $1$.
\end{theorem}

The optimal constants of the previous inequalities are essentially unknown;
these depend of the chosen scalar field. One of the few cases in which the
optimal constants are known is the case of the mixed $\left(  \ell_{\frac
{p}{p-1}},\ell_{2}\right)  $-Littlewood inequality (see \cite{racsam, pell,
PT}):

\begin{theorem}
[Mixed $\left(  \ell_{\frac{p}{p-1}},\ell_{2}\right)  $-Littlewood
inequality]Let $p\geq2$. There is a constant $C_{p,\infty}$ such that
\[
\left(  \sum_{i=1}^{\infty}\left(  \sum_{j=1}^{\infty}|T(e_{i},e_{j}%
)|^{2}\right)  ^{\frac{1}{2}\frac{p}{p-1}}\right)  ^{\frac{p-1}{p}}\leq
C_{p,\infty}\Vert T\Vert,
\]
for all continuous bilinear forms $T:X_{p}\times X_{\infty}\rightarrow
\mathbb{R}$. Moreover, $2^{\frac{1}{2}-\frac{1}{p}}\leq C_{p,\infty}%
\leq\mathrm{A}_{\frac{p}{p-1}}$. Particularly, if $p_{0}$ is the number
defined in (\ref{pezero}) then the optimal constant $C_{p,\infty}$ is
$2^{\frac{1}{2}-\frac{1}{p}}$, for all $p\geq\frac{p_{0}}{p_{0}-1}%
\approx2.18006$.
\end{theorem}

The Hardy--Littlewood inequalities for bilinear forms in $\ell_{p}$ spaces
were proved in 1934 \cite{hardy}; the original proofs of Hardy and Littlewood
rely on a result proved by Littlewood that, in general terms, is none other
than the Kintchine inequality. The result, in modern mathematical notation,
given by Littlewood is the following:

\begin{theorem}
[\cite{LLL}, pag. 170]Let $0<\rho<\infty$ be a real number. Then
\[
\frac{\left(  2^{-N}\sum_{\beta\in\left\{  -1,1\right\}  ^{N}}\left\vert
\sum_{j=1}^{N}a_{j}\beta_{j}\right\vert ^{\rho}\right)  ^{\frac{1}{\rho}}%
}{\left(  \sum_{j=1}^{N}|a_{j}|^{2}\right)  ^{\frac{1}{2}}}%
\]
lies between two constants $A_{\rho}$ and $B_{\rho}$ (depending only on $\rho
$), whatever is the scalar sequence $(a_{j})_{j=1}^{N}\in\mathbb{K}^{N}$ and
whatever is the value of $N$.
\end{theorem}

The Hardy--Littlewood inequalities consist in optimal extensions of
Littlewood's $4/3$ inequality \cite{LLL} (originally stated for $c_{0}$
spaces). In the last years the interest in this subject (which can be
considered part of the theory of multiple summing and absolutely summing
operators) was renewed and several authors became interested in this field
(\cite{abps, ara2, ap, mic, wast, pell, PT, pilar}).

Our first aim in this work is to show that the Khintchine inequality and the
mixed $\left(  \ell_{\frac{p}{p-1}},\ell_{2}\right)  $-Littlewood inequality
are equivalent; it means that one can be obtained from the other one. This
assertion not only works in the way we have stated, but as we will see later,
the multiple Khintchine inequality and the multilinear mixed $\left(
\ell_{\frac{p}{p-1}},\ell_{2}\right)  $-Littlewood inequality (see
\cite{racsam}) are also equivalent.

The Khintchine inequality and the multiple Khintchine inequality are very
useful tools in the theory of absolutely summing operators (even in its non
linear extensions) and related classical inequalities. We stress, for
instance, the striking advances in the estimates of the Hardy--Littlewood
constants (see \cite{ap}). Our second aim in this work is to estimate the
optimal constants in the multiple Khintchine inequality. We completely solve
this issue in the Section \ref{multi} and as application, in the final
section, we obtain the optimal constants of the multilinear mixed $\left(
\ell_{\frac{p}{p-1}},\ell_{2}\right)  $-Littlewood inequality, completing the
recent estimates in \cite{racsam, pell}.

\section{The Khintchine inequality is equivalent to the mixed $\left(
\ell_{\frac{p}{p-1}},\ell_{2}\right)  $-Littlewood inequality\label{classic}}

In this section we prove the equivalence between two classical inequalities
which have been mentioned on the last section. Moreover, we extract from this
equivalence the equality $C_{p,\infty}=\mathrm{A}_{\frac{p}{p-1}}$, for all
$p\in\lbrack2,\infty].$ This estimate complements, in the bilinear case, the
paper \cite{racsam}. We emphasize which the seminal ideas of this general
equivalency was given in the book \cite[Chapter 1]{blei}.

Let us start showing how a proof of the mixed $\left(  \ell_{\frac{p}{p-1}%
},\ell_{2}\right)  $-Littlewood inequality can be made from the Kintchine
inequality. This fact is known in this field and we shall include a short
proof (following the ideas of \cite[Theorem 1.2]{racsam}) for the sake of completeness.

Henceforth, for all $n\in\mathbb{N}$, $X_{p}^{n}$ will denote the finite
dimensional Banach space $\mathbb{R}^{n}$ endowed with the $\ell_{p}$ norm.

\begin{theorem}
[Mixed $\left(  \ell_{\frac{p}{p-1}},\ell_{2}\right)  $-Littlewood
inequality]\label{ppp}Let $p\geq2$. There is a constant $C_{p,\infty}$ such
that
\[
\left(  \sum_{i=1}^{\infty}\left(  \sum_{j=1}^{\infty}|T(e_{i},e_{j}%
)|^{2}\right)  ^{\frac{1}{2}\frac{p}{p-1}}\right)  ^{\frac{p-1}{p}}\leq
C_{p,\infty}\Vert T\Vert,
\]
for all continuous bilinear forms $T:X_{p}\times X_{\infty}\rightarrow
\mathbb{R}$. Moreover, $C_{p,\infty}\leq\mathrm{A}_{p^{\ast}}$.
\end{theorem}

\begin{proof}
Let $n$ be a positive integer and $T:X_{p}^{n}\times X_{\infty}^{n}%
\rightarrow\mathbb{R}$ be a bilinear form. Then, invoking the Khintchine
inequality, we have
\begin{align*}
\left(  \sum\limits_{i=1}^{n}\left(  \sum\limits_{j=1}^{n}\left\vert T\left(
e_{i},e_{j}\right)  \right\vert ^{2}\right)  ^{\frac{1}{2}\frac{p}{p-1}%
}\right)  ^{\frac{p-1}{p}}  &  \leq\mathrm{A}_{\frac{p}{p-1}}\left(
\sum\limits_{i=1}^{n}\int_{0}^{1}\left\vert \sum\limits_{j=1}^{n}%
r_{j}(t)T\left(  e_{i},e_{j}\right)  \right\vert ^{\frac{p}{p-1}}dt\right)
^{\frac{p-1}{p}}\\
&  =\mathrm{A}_{\frac{p}{p-1}}\left(  \int_{0}^{1}\sum\limits_{i=1}%
^{n}\left\vert T\left(  e_{i},\sum\limits_{j=1}^{n}r_{j}(t)e_{j}\right)
\right\vert ^{\frac{p}{p-1}}dt\right)  ^{\frac{p-1}{p}}\\
&  \leq\mathrm{A}_{\frac{p}{p-1}}\left(  \int_{0}^{1}\left\Vert T\left(
\cdot,\sum\limits_{j=1}^{n}r_{j}(t)e_{j}\right)  \right\Vert ^{\frac{p}{p-1}%
}dt\right)  ^{\frac{p-1}{p}}\\
&  \leq\mathrm{A}_{\frac{p}{p-1}}\sup_{t\in\lbrack0,1]}\left\Vert T\left(
\cdot,\sum\limits_{j=1}^{n}r_{j}(t)e_{j}\right)  \right\Vert \\
&  \leq\mathrm{A}_{\frac{p}{p-1}}\left\Vert T\right\Vert .
\end{align*}
As follows, the theorem is proved and $C_{p,\infty}\leq\mathrm{A}_{p^{\ast}}.$
\end{proof}

The following well-known lemma, credited to Hermann Minkowski, will be useful
along this paper (see \cite[Corollary 5.4.2]{garling}):\bigskip

\begin{lemma}
\label{mink.seq} For $0<p<q<+\infty$, and any sequence of scalars $\left(
a_{ij}\right)  _{i,j\in\mathbb{N}}$ we have
\[
\left(  \sum_{i}\left(  \sum_{j}|a_{ij}|^{p}\right)  ^{\frac{q}{p}}\right)
^{\frac{1}{q}}\leq\left(  \sum_{j}\left(  \sum_{i}|a_{ij}|^{q}\right)
^{\frac{p}{q}}\right)  ^{\frac{1}{p}}.
\]

\end{lemma}

Now, let us recall that using Lemma \ref{mink.seq}, the mixed $\left(
\ell_{\frac{p}{p-1}},\ell_{2}\right)  $-Littlewood inequality implies the next result:

\begin{theorem}
\label{pppp} Let $p\geq2$. There is a constant $D_{p,\infty}$ such that
\[
\left(  \sum_{j=1}^{\infty}\left(  \sum_{i=1}^{\infty}|T(e_{i},e_{j}%
)|^{\frac{p}{p-1}}\right)  ^{2\frac{p-1}{p}}\right)  ^{\frac{1}{2}}\leq
D_{p,\infty}\Vert T\Vert,
\]
for all continuous bilinear forms $T:X_{p}\times X_{\infty}\rightarrow
\mathbb{R}$. Moreover, $D_{p,\infty}\leq C_{p,\infty}.$
\end{theorem}

\begin{proof}
It is enough to note that $\frac{p}{p-1}\leq2$, thus by Lemma \ref{mink.seq}
and the Theorem \ref{ppp} we obtain
\begin{align*}
\left(  \sum_{j=1}^{\infty}\left(  \sum_{i=1}^{\infty}|T(e_{i},e_{j}%
)|^{\frac{p}{p-1}}\right)  ^{2\frac{p-1}{p}}\right)  ^{\frac{1}{2}}  &
\leq\left(  \sum_{i=1}^{\infty}\left(  \sum_{j=1}^{\infty}|T(e_{i},e_{j}%
)|^{2}\right)  ^{\frac{1}{2}\frac{p}{p-1}}\right)  ^{\frac{p-1}{p}}\\
&  \leq C_{p,\infty}\Vert T\Vert.
\end{align*}
Then, the theorem is proved with the estimate $D_{p,\infty}\leq C_{p,\infty}.$
\end{proof}

\bigskip

In this way, we have proved that Khintchine inequality implies Theorem
\ref{ppp} as well as the Theorem \ref{ppp} implies Theorem \ref{pppp}.
Furthermore, $D_{p,\infty}\leq C_{p,\infty}\leq\mathrm{A}_{p^{\ast}}.$ Then,
if Theorem \ref{pppp} implies the Khintchine inequality, we will have that the
mixed $\left(  \ell_{\frac{p}{p-1}},\ell_{2}\right)  $-Littlewood inequality
and the Khintchine inequality are equivalent. That is the assertion in the following:

\begin{theorem}
\label{implies}Theorem (\ref{pppp}) implies Khintchine inequality (the case in
which $1\leq p\leq2$). Moreover, $\mathrm{A}_{p}\leq D_{p^{\ast},\infty}$.
\end{theorem}

\begin{proof}
Let $p\in\left[  1,2\right]  $, $n\in\mathbb{N}$, and $\left(  a_{j}\right)
_{j=1}^{n}$ a scalar sequence.

By (\ref{equal}), we know that for any choice of signs $\left(  \varepsilon
_{j}\right)  _{j=1}^{n}\in\{-1,1\}^{n}$ we have
\[
\left(  \int_{0}^{1}\left\vert \sum_{j=1}^{n}a_{j}r_{j}(t)\right\vert
^{p}dt\right)  ^{\frac{1}{p}}=\left(  \int_{0}^{1}\left\vert \sum_{j=1}%
^{n}a_{j}r_{j}(t)\varepsilon_{j}\right\vert ^{p}dt\right)  ^{\frac{1}{p}}.
\]
Now, solving the integral on the right hand we obtain%
\begin{align}
\left(  \int_{0}^{1}\left\vert \sum_{j=1}^{n}a_{j}r_{j}(t)\right\vert
^{p}\right)  ^{\frac{1}{p}}  &  =\left(  \left(  \frac{1}{2}\right)  ^{n}%
\sum_{\beta\in\left\{  -1,1\right\}  ^{n}}\left\vert \sum_{j=1}^{n}a_{j}%
\beta_{j}\varepsilon_{j}\right\vert ^{p}\right)  ^{\frac{1}{p}}\label{este}\\
&  =\left(  \left(  \frac{1}{2}\right)  ^{n}\sum_{i=1}^{2^{n}}\left\vert
\sum_{j=1}^{n}a_{j}\delta_{j}^{\left(  i\right)  }\varepsilon_{j}\right\vert
^{p}\right)  ^{\frac{1}{p}}\nonumber
\end{align}
where each $\delta_{j}^{\left(  i\right)  }$ is $1$ or $-1$, for all $i,j$.

Hence, if $\left(  x_{i}\right)  _{i=1}^{\infty}\in B_{X_{p^{\ast}}}$, by
H\"{o}lder's inequality and (\ref{este}) we have
\begin{align}
\left\vert \left(  \frac{1}{2}\right)  ^{\frac{n}{p}}\sum_{i=1}^{2^{n}}%
\sum_{j=1}^{n}a_{j}\delta_{j}^{\left(  i\right)  }\varepsilon_{j}%
x_{i}\right\vert  &  \leq\left(  \left(  \frac{1}{2}\right)  ^{n}\sum
_{i=1}^{2^{n}}\left\vert \sum_{j=1}^{n}a_{j}\delta_{j}^{\left(  i\right)
}\varepsilon_{j}\right\vert ^{p}\right)  ^{\frac{1}{p}}\left(  \sum
_{i=1}^{2^{n}}\left\vert x_{i}\right\vert ^{p^{\ast}}\right)  ^{\frac
{1}{p^{\ast}}}\label{este2}\\
&  =\left(  \int_{0}^{1}\left\vert \sum_{j=1}^{n}a_{j}r_{j}(t)\right\vert
^{p}\right)  ^{\frac{1}{p}}\left\Vert \left(  x_{i}\right)  _{i=1}^{2^{n}%
}\right\Vert _{p^{\ast}}\nonumber\\
&  \leq\left(  \int_{0}^{1}\left\vert \sum_{j=1}^{n}a_{j}r_{j}(t)\right\vert
^{p}\right)  ^{\frac{1}{p}}.\nonumber
\end{align}
Now, define the bilinear form $A:X_{p^{\ast}}^{2^{n}}\times X_{\infty}%
^{n}\rightarrow\mathbb{R}$, such that $A\left(  e_{i},e_{j}\right)  =\left(
\frac{1}{2}\right)  ^{\frac{n}{p}}a_{j}\delta_{j}^{\left(  i\right)  }$, for
$1\leq j\leq n$, $1\leq i\leq2^{n}$. Clearly $A$ is bounded and thus
\begin{align*}
&  \left(  \sum_{j=1}^{n}\left\vert a_{j}\right\vert ^{2}\right)  ^{\frac
{1}{2}}\\
&  =\left(  \sum_{j=1}^{n}\left(  \sum_{i=1}^{2^{n}}\left\vert \left(
\frac{1}{2}\right)  ^{\frac{n}{p}}\delta_{j}^{\left(  i\right)  }%
a_{j}\right\vert ^{p}\right)  ^{\frac{2}{p}}\right)  ^{\frac{1}{2}}\\
&  =\left(  \sum_{j=1}^{n}\left(  \sum_{i=1}^{2^{n}}\left\vert A\left(
e_{i},e_{j}\right)  \right\vert ^{p}\right)  ^{\frac{2}{p}}\right)  ^{\frac
{1}{2}}\\
&  \leq D_{p^{\ast},\infty}\Vert A\Vert\\
&  =D_{p^{\ast},\infty}\sup\left\{  \left\vert \left(  \frac{1}{2}\right)
^{\frac{n}{p}}\sum_{i=1}^{2^{n}}\sum_{j=1}^{n}a_{j}\delta_{j}^{\left(
i\right)  }\varepsilon_{j}x_{i}\right\vert :\left\Vert x\right\Vert
_{X_{p^{\ast}}}\leq1,~\varepsilon_{j}\in\{-1,1\}\right\}  \\
&  \leq D_{p^{\ast},\infty}\left(  \int_{0}^{1}\left\vert \sum_{j=1}^{n}%
a_{j}r_{j}(t)\right\vert ^{p}\right)  ^{\frac{1}{p}},
\end{align*}
where, the last inequality stands by (\ref{este2}).

As follows, the theorem is proved and $\mathrm{A}_{p}\leq D_{p^{\ast},\infty}$.
\end{proof}

\bigskip

Note that we have proved that for all $p\in\left[  1,2\right]  $%
\[
D_{p^{\ast},\infty}=C_{p^{\ast},\infty}=\mathrm{A}_{p},
\]
or, equivalently, for all $p\in\left[  2,\infty\right]  $%
\[
D_{p,\infty}=C_{p,\infty}=\mathrm{A}_{\frac{p}{p-1}}.
\]

\bigskip

As announced, this recovers and completes the estimates of the papers
\cite{racsam, pell} for the bilinear case.

\section{\bigskip Optimal constants in the multiple Khintchine
inequality\label{multi}}

There are several extensions of the Khintchine inequality, some of them deal
with higher dimensions. In this section we will deal with a very important
extension, the so-called Khintchine inequality for multiple sums, or multiple
Khintchine inequality, (see \cite[pag. 455]{defloret} and the references therein):

\begin{theorem}
\label{multikhin}Let $0<r<\infty$, $m\in%
\mathbb{N}
$, and $\left(  y_{i_{1}...i_{m}}\right)  _{i_{1,}...,i_{m}=1}^{N}$ an array
of scalars. There is a constant $K_{m,r}\geq1$, such that
\begin{equation}
\left(  \sum_{i_{1},...,i_{m}=1}^{N}\left\vert y_{i_{1}...i_{m}}\right\vert
^{2}\right)  ^{\frac{1}{2}}\leq K_{m,r}\left(
{\displaystyle\int\limits_{I^{m}}}
\left\vert \sum_{i_{1},...,i_{m}=1}^{N}r_{i_{1}}\left(  t_{1}\right)
...r_{i_{m}}\left(  t_{m}\right)  y_{i_{1}...i_{m}}\right\vert ^{r}%
dt_{1}...dt_{m}\right)  ^{\frac{1}{r}} \label{popa}%
\end{equation}
for all $N\in%
\mathbb{N}
$, where $r_{i_{j}}\left(  t_{j}\right)  $ are denoting the Rademacher
functions, for all $j\in\left\{  1,...,m\right\}  $ and $i_{j}\in\left\{
1,...,N\right\}  .$ Moreover, $K_{m,r}\leq\left(  \mathrm{A}_{r}\right)  ^{m}$.
\end{theorem}

In \cite{defloret} was given the details for $m=2$, but these arguments
contain all the elements of the general case. For the sake of completeness, we
give an elementary proof:

\begin{proof}
The proof will be obtained by induction in $m$. The case $m=1$ is exactly the
Khintchine inequality. Let us start the proof for the case $0<r\leq2$. Assume
inductively the result holds for $m-1$, then%
\begin{align}
&  \left(  \sum_{i_{1},...,i_{m}=1}^{N}\left\vert y_{i_{1}...i_{m}}\right\vert
^{2}\right)  ^{\frac{1}{2}}\nonumber\\
&  =\left(  \sum_{i_{1}=1}^{N}\left(  \sum_{i_{2},...,i_{m}=1}^{N}\left\vert
y_{i_{1}...i_{m}}\right\vert ^{2}\right)  ^{\frac{1}{2}2}\right)  ^{\frac
{1}{2}}\label{mink1}\\
&  \leq\left(  \mathrm{A}_{r}\right)  ^{m-1}\left(  \sum_{i_{1}=1}^{N}\left(
{\displaystyle\int\limits_{I^{m-1}}}
\left\vert \sum_{i_{2},...,i_{m}=1}^{N}r_{i_{2}}\left(  t_{2}\right)
...r_{i_{m}}\left(  t_{m}\right)  y_{i_{1}...i_{m}}\right\vert ^{r}%
dt_{2}...dt_{m}\right)  ^{\frac{2}{r}}\right)  ^{\frac{1}{2}}.\nonumber
\end{align}
Since $\frac{2}{r}\geq1$, using the Minkowski integral inequality we get%
\begin{align}
&  \left(
{\displaystyle\sum\limits_{i_{1}=1}^{N}}
\left(
{\displaystyle\int\limits_{I^{m-1}}}
\left\vert \sum_{i_{2},...,i_{m}=1}^{N}r_{i_{2}}\left(  t_{2}\right)
...r_{i_{m}}\left(  t_{m}\right)  y_{i_{1}...i_{m}}\right\vert ^{r}%
dt_{2},...,dt_{m}\right)  ^{\frac{2}{r}}\right)  ^{\frac{1}{2}}\label{mink2}\\
&  \leq\left[
{\displaystyle\int\limits_{I^{m-1}}}
\left(
{\displaystyle\sum\limits_{i_{1}=1}^{N}}
\left\vert \sum_{i_{2},...,i_{m}=1}^{N}r_{i_{2}}\left(  t_{2}\right)
...r_{i_{m}}\left(  t_{m}\right)  y_{i_{1}...i_{m}}\right\vert ^{r.\frac{2}%
{r}}\right)  ^{\frac{r}{2}}dt_{2}...dt_{m}\right]  ^{\frac{1}{r}}.\nonumber
\end{align}
From (\ref{mink1}) and (\ref{mink2}), we have%
\begin{align*}
&  \left(  \sum_{i_{1},...,i_{m}=1}^{N}\left\vert \left(  y_{i_{1}...i_{m}%
}\right)  \right\vert ^{2}\right)  ^{\frac{1}{2}}\\
&  \leq\left(  \mathrm{A}_{r}\right)  ^{m-1}\left[
{\displaystyle\int\limits_{I^{m-1}}}
\left(
{\displaystyle\sum\limits_{i_{1}=1}^{N}}
\left\vert \sum_{i_{2},...,i_{m}=1}^{N}r_{i_{2}}\left(  t_{2}\right)
...r_{i_{m}}\left(  t_{m}\right)  y_{i_{1}...i_{m}}\right\vert ^{2}\right)
^{\frac{r}{2}}dt_{2}...dt_{m}\right]  ^{\frac{1}{r}}.
\end{align*}
Now, Khintchine inequality furnishes us the inequality%
\begin{align*}
&  \left(
{\displaystyle\sum\limits_{i_{1}=1}^{N}}
\left\vert \sum_{i_{2},...,i_{m}=1}^{N}r_{i_{2}}\left(  t_{2}\right)
...r_{i_{m}}\left(  t_{m}\right)  y_{i_{1}...i_{m}}\right\vert ^{2}\right)
^{\frac{1}{2}}\\
&  \leq\mathrm{A}_{r}\left(
{\displaystyle\int\limits_{0}^{1}}
\left\vert
{\displaystyle\sum\limits_{i_{1}=1}^{N}}
r_{i_{1}}\left(  t_{1}\right)  \sum_{i_{2},...,i_{m}=1}^{N}r_{i_{2}}\left(
t_{2}\right)  ...r_{i_{m}}\left(  t_{m}\right)  y_{i_{1}...i_{m}}\right\vert
^{r}dt_{1}\right)  ^{\frac{1}{r}},
\end{align*}
thus%
\begin{align*}
&  \left(  \sum_{i_{1},...,i_{m}=1}^{N}\left\vert \left(  y_{i_{1}...i_{m}%
}\right)  \right\vert ^{2}\right)  ^{\frac{1}{2}}\\
&  \leq\left(  \mathrm{A}_{r}\right)  ^{m-1}\left[
{\displaystyle\int\limits_{I^{m-1}}}
\mathrm{A}_{r}\left(
{\displaystyle\int\limits_{0}^{1}}
\left\vert
{\displaystyle\sum\limits_{i_{1}=1}^{N}}
r_{i_{1}}\left(  t_{1}\right)  \sum_{i_{2},...,i_{m}=1}^{N}r_{i_{2}}\left(
t_{2}\right)  ...r_{i_{m}}\left(  t_{m}\right)  y_{i_{1}...i_{m}}\right\vert
^{r}dt_{1}\right)  ^{\frac{r}{r}}dt_{2}...dt_{m}\right]  ^{\frac{1}{r}}\\
&  =\left(  \mathrm{A}_{r}\right)  ^{m}\left[
{\displaystyle\int\limits_{I^{m-1}}}
\left(
{\displaystyle\int\limits_{0}^{1}}
\left\vert \sum_{i_{1},...,i_{m}=1}^{N}r_{i_{1}}\left(  t_{1}\right)
...r_{i_{m}}\left(  t_{m}\right)  y_{i_{1}...i_{m}}\right\vert ^{r}%
dt_{1}\right)  ^{\frac{r}{r}}dt_{2}...dt_{m}\right]  ^{\frac{1}{r}},
\end{align*}
and by Fubini's theorem%
\[
\left(  \sum_{i_{1},...,i_{m}=1}^{N}\left\vert \left(  y_{i_{1}...i_{m}%
}\right)  \right\vert ^{2}\right)  ^{\frac{1}{2}}\leq\left(  \mathrm{A}%
_{r}\right)  ^{m}\left[
{\displaystyle\int\limits_{I^{m}}}
\left\vert
{\displaystyle\sum\limits_{i_{1},...,i_{m}=1}^{N}}
r_{i_{1}}\left(  t_{1}\right)  r_{i_{2}}\left(  t_{2}\right)  ...r_{i_{m}%
}\left(  t_{m}\right)  y_{i_{1}...i_{m}}\right\vert ^{r}dt_{1}...dt_{m}%
\right]  ^{\frac{1}{r}}.
\]
On the other hand, since particularly we have proved
\[
\left(  \sum_{i_{1},...,i_{m}=1}^{N}\left\vert \left(  y_{i_{1}...i_{m}%
}\right)  \right\vert ^{2}\right)  ^{\frac{1}{2}}\leq\left[
{\displaystyle\int\limits_{I^{m}}}
\left\vert
{\displaystyle\sum\limits_{i_{1},...,i_{m}=1}^{N}}
r_{i_{1}}\left(  t_{1}\right)  r_{i_{2}}\left(  t_{2}\right)  ...r_{i_{m}%
}\left(  t_{m}\right)  y_{i_{1}...i_{m}}\right\vert ^{2}dt_{1}...dt_{m}%
\right]  ^{\frac{1}{2}},
\]
the assertion in the case $2<r<\infty$ follows trivially by the norm inclusion
between the $L_{p}$ spaces.
\end{proof}

Theorem \ref{multikhin} has important applications, for example, in multiple
summing operator theory. In fact, a frequent application of the Theorem
\ref{multikhin} can be traced in modern proofs of the Hardy--Littlewood
inequalities (see \cite{abps, ap, pell}).

Recently, Pellegrino \textit{et.al.} (see \cite[Proposition 1]{pel}) showed
that for all $m\in%
\mathbb{N}
$, $K_{m,1}=\left(  \sqrt{2}\right)  ^{m}=\left(  \mathrm{A}_{1}\right)  ^{m}%
$. Our goal in this section is to prove that, in fact, the optimal constants
$K_{m,r}$ are $\left(  \mathrm{A}_{r}\right)  ^{m}$, for all $m\in%
\mathbb{N}
$ and for all $0<r<\infty.$ In order to achieve it, we need an auxiliary
result proved in \cite[Proposition 1]{popa2}:

\begin{proposition}
\label{multipopa}Let $0<r<\infty$, $Z$ be a Banach space, $m,N_{1}%
,...,N_{m}\in%
\mathbb{N}
$, and $\left(  z_{i_{1}...i_{m}}\right)  _{i_{1,}...,i_{m}=1}^{N_{1}%
,...,N_{m}}\subset Z$. Then%
\begin{align*}
&  \left(
{\displaystyle\int\limits_{I^{m}}}
\left\Vert \sum_{i_{1},...,i_{m}=1}^{N_{1},...,N_{m}}r_{i_{1}}\left(
t_{1}\right)  ...r_{i_{m}}\left(  t_{m}\right)  z_{i_{1}...i_{m}}\right\Vert
^{r}dt_{1}...dt_{m}\right)  ^{\frac{1}{r}}\\
&  =\left(  \frac{1}{2^{N_{1}+\cdots+N_{m}}}\sum_{\left(  \eta^{\left(
1\right)  },...,\eta^{\left(  m\right)  }\right)  \in D_{N_{1}}\times
\cdots\times D_{N_{m}}}\left\Vert \sum_{i_{1},...,i_{m}=1}^{N_{1},...,N_{m}%
}\eta_{i_{1}}^{\left(  1\right)  },...,\eta_{i_{m}}^{\left(  m\right)
}z_{i_{1}...i_{m}}\right\Vert ^{r}\right)  ^{\frac{1}{r}},
\end{align*}
where $\eta^{\left(  j\right)  }=\left(  \eta_{1}^{\left(  j\right)
},...,\eta_{N_{j}}^{\left(  j\right)  }\right)  \in D_{N_{j}}$ for all
$j\in\left\{  1,...,m\right\}  $, and for all $n\in\mathbb{N}$ the set
$D_{n}:=\left\{  -1,1\right\}  ^{n}$.
\end{proposition}

Another result we will use, due to Haagerup in (\cite{Ha}), involving the
Central Limit Theorem is:

\begin{lemma}
\label{zentral}Let $r\in\left(  0,2\right)  $. For each $n\in%
\mathbb{N}
$ consider $\left(  a_{1},...,a_{n}\right)  =\frac{1}{\sqrt{n}}\left(
1,...,1\right)  $. Then
\[
\lim_{n\rightarrow\infty}\left(  \int_{0}^{1}\left\vert \sum_{j=1}^{n}%
a_{j}r_{j}(t)\right\vert ^{r}dt\right)  ^{\frac{1}{r}}=\sqrt{2}\left(
\frac{\Gamma\left(  \frac{r+1}{2}\right)  }{\sqrt{\pi}}\right)  ^{\frac{1}{r}%
}\text{,}%
\]
where $r_{j}\left(  t\right)  $ denotes the $j-$th Rademacher function for all
$j$.
\end{lemma}

\begin{theorem}
\label{multik}For all $0<r<\infty$ and $m\in%
\mathbb{N}
,$ the optimal constant $K_{m,r}$ in (\ref{popa}) is $\left(  \mathrm{A}%
_{r}\right)  ^{m}$.
\end{theorem}

\begin{proof}
From Theorem \ref{multikhin} we already know that $K_{m,r}\leq\left(
\mathrm{A}_{r}\right)  ^{m}$ for all $0<r<\infty$ and $m\in%
\mathbb{N}
$, so we only need to check the other inequality.

Obviously $K_{m,r}\geq\left(  \mathrm{A}_{r}\right)  ^{m}=\left(  1\right)
^{m}=1,$ for all $2\leq r<\infty$. Let us to prove the case $0<r<2$. Firstly,
let $0<r\leq p_{0}$, where $p_{0}$ is the number defined in (\ref{pezero}). We
start the proof by showing the case $m=2$, i.e., $K_{2,r}\geq\left(
2^{\frac{1}{r}-\frac{1}{2}}\right)  ^{2}=2^{\frac{2-r}{r}}$.

If we take%
\[
\left(  y_{ij}\right)  _{i,j=1}^{N}=\left(
\begin{array}
[c]{ccccc}%
1 & 1 & 0 & \cdots & 0\\
1 & 1 & 0 & \cdots & 0\\
0 & 0 & 0 & \cdots & 0\\
\vdots & \vdots & \vdots & \ddots & \vdots\\
0 & 0 & 0 & \cdots & 0
\end{array}
\right)  ,
\]
we have%
\[
\left(  \sum_{i,j=1}^{N}\left\vert \left(  y_{ij}\right)  \right\vert
^{2}\right)  ^{\frac{1}{2}}=\left(  4\right)  ^{\frac{1}{2}}=2.
\]
On the other side, from Proposition \ref{multipopa} we have%
\begin{align*}
&  \left(
{\displaystyle\int\limits_{I^{2}}}
\left\vert \sum_{i,j=1}^{N}r_{i}\left(  t\right)  r_{j}\left(  s\right)
y_{ij}\right\vert ^{r}dtds\right)  ^{\frac{1}{r}}\\
&  =\left(
{\displaystyle\int\limits_{I^{2}}}
\left\vert \sum_{i,j=1}^{2}r_{i}\left(  t\right)  r_{j}\left(  s\right)
\right\vert ^{r}dtds\right)  ^{\frac{1}{r}}\\
&  =\left(  \frac{1}{2^{4}}\sum_{\left(  \eta^{\left(  1\right)  }%
,\eta^{\left(  2\right)  }\right)  \in D_{N}\times D_{N}}\left\vert
\sum_{i,j=1}^{2}\eta_{i}^{\left(  1\right)  }\eta_{j}^{\left(  2\right)
}\right\vert ^{r}\right)  ^{\frac{1}{r}}\\
&  =\left(  \frac{1}{2^{4}}\left(  \sum_{\eta^{\left(  1\right)  }\in D_{N}%
}\left\vert \sum_{i=1}^{2}\eta_{i}^{\left(  1\right)  }\right\vert
^{r}\right)  \left(  \sum_{\eta^{\left(  2\right)  }\in D_{N}}\left\vert
\sum_{j=1}^{2}\eta_{j}^{\left(  2\right)  }\right\vert ^{r}\right)  \right)
^{\frac{1}{r}}\\
&  =\left(  2^{-4}\left(  2\cdot2^{r}\right)  \left(  2\cdot2^{r}\right)
\right)  ^{\frac{1}{r}}\\
&  =2^{\frac{2r-2}{r}}.
\end{align*}
In that case%
\begin{align*}
K_{2,r}  &  \geq\frac{\left(  \sum_{i,j=1}^{N}\left\vert \left(
y_{ij}\right)  \right\vert ^{2}\right)  ^{\frac{1}{2}}}{\left(
{\displaystyle\int\limits_{I^{2}}}
\left\vert \sum_{i,j=1}^{N}r_{i}\left(  t\right)  r_{j}\left(  s\right)
y_{ij}\right\vert ^{r}dtds\right)  ^{\frac{1}{r}}}\\
&  =\frac{2}{2^{^{\frac{2r-2}{r}}}}=2^{\frac{2-r}{r}}.
\end{align*}
For the general case $m$, we consider
\begin{equation}
y_{i_{1}...i_{m}}=\left\{
\begin{array}
[c]{c}%
1;i_{1},...,i_{m}\in\left\{  1,2\right\} \\
0;i_{1},...,i_{m}\in\left\{  3,N\right\}  .
\end{array}
\right.  \label{menor}%
\end{equation}
Then,
\[
\left(  \sum_{i_{1},...,i_{m}=1}^{N}\left\vert \left(  y_{i_{1}...i_{m}%
}\right)  \right\vert ^{2}\right)  ^{\frac{1}{2}}=2^{\frac{m}{2}}%
\]
and%
\[
\left(
{\displaystyle\int\limits_{I^{m}}}
\left\vert \sum_{i_{1},...,i_{m}=1}^{N}r_{i_{1}}\left(  t_{1}\right)
...r_{i_{m}}\left(  t_{m}\right)  y_{i_{1}...i_{m}}\right\vert ^{r}%
dt_{1}...dt_{m}\right)  ^{\frac{1}{r}}=2^{\frac{mr-m}{r}}.
\]
In this way%
\[
K_{m,r}\geq\frac{2^{\frac{m}{2}}}{2^{\frac{mr-m}{r}}}=\left(  2^{\frac
{2-r}{2r}}\right)  ^{m},
\]
i.e., the assertion is proved for $0<r\leq p_{0}$ and $m\in%
\mathbb{N}
$.

On the other hand, let $r\in\left(  p_{0},2\right)  $ and $m\in%
\mathbb{N}
$. For each $N\in%
\mathbb{N}
$ consider
\[
y_{i_{1}...i_{m}}=\left\{  \frac{1}{N^{\frac{m}{2}}};i_{1},...,i_{m}%
\in\left\{  1,N\right\}  \right\}  .
\]
We have%
\[
\left(  \sum_{i_{1},...,i_{m}=1}^{N}\left\vert \left(  y_{i_{1}...i_{m}%
}\right)  \right\vert ^{2}\right)  ^{\frac{1}{2}}=\left(  \sum_{i_{1}%
,...,i_{m}=1}^{N}\frac{1}{N^{m}}\right)  ^{\frac{1}{2}}=1,
\]
and, at the same time, from Proposition \ref{multipopa} we have
\begin{align*}
&
{\displaystyle\int\limits_{I^{m}}}
\left\vert \sum_{i_{1},...,i_{m}=1}^{N}r_{i_{1}}\left(  t_{1}\right)
...r_{i_{m}}\left(  t_{m}\right)  y_{i_{1}...i_{m}}\right\vert ^{r}%
dt_{1}...dt_{m}\\
&  =\frac{1}{2^{mN}}\sum_{\left(  \eta^{\left(  1\right)  },...,\eta^{\left(
m\right)  }\right)  \in D_{N}\times\cdots\times D_{N}}\left\vert \sum
_{i_{1},...,i_{m}=1}^{N}\frac{\eta_{i_{1}}^{\left(  1\right)  },...,\eta
_{i_{m}}^{\left(  m\right)  }}{N^{\frac{m}{2}}}\right\vert ^{r}\\
&  =\frac{1}{2^{mN}}\frac{1}{N^{\frac{mr}{2}}}\sum_{\left(  \eta^{\left(
1\right)  },...,\eta^{\left(  m\right)  }\right)  \in D_{N}\times\cdots\times
D_{N}}\left\vert \sum_{i_{1},...,i_{m}=1}^{N}\eta_{i_{1}}^{\left(  1\right)
},...,\eta_{i_{m}}^{\left(  m\right)  }\right\vert ^{r}\\
&  =\frac{1}{2^{mN}}\frac{1}{N^{\frac{mr}{2}}}\left(  \sum_{\eta^{\left(
1\right)  }\in D_{N}}\left\vert \sum_{i_{1}=1}^{N}\eta_{i_{1}}^{\left(
1\right)  }\right\vert ^{r}\right)  \cdot\cdots\cdot\left(  \sum
_{\eta^{\left(  m\right)  }\in D_{N}}\left\vert \sum_{i_{m}=1}^{N}\eta_{i_{m}%
}^{\left(  m\right)  }\right\vert ^{r}\right) \\
&  =\prod_{j=1}^{m}\left(  \frac{1}{2^{N}}\sum_{\eta^{\left(  j\right)  }\in
D_{N}}\left\vert \sum_{i_{j}=1}^{N}\eta_{i_{j}}^{\left(  j\right)  }\frac
{1}{N^{\frac{1}{2}}}\right\vert ^{r}\right) \\
&  =\prod_{j=1}^{m}\left(
{\displaystyle\int\limits_{I}}
\left\vert \sum_{i_{j}=1}^{N}r_{i_{j}}\left(  t_{j}\right)  a_{i_{j}%
}\right\vert ^{r}dt_{j}\right)
\end{align*}
where $\left(  a_{1},...,a_{N}\right)  =\frac{1}{\sqrt{N}}\left(
1,...,1\right)  $. Using the Lemma \ref{zentral}, we have%
\begin{align*}
&  \lim_{N\rightarrow\infty}\left(
{\displaystyle\int\limits_{I^{m}}}
\left\vert \sum_{i_{1},...,i_{m}=1}^{N}r_{i_{1}}\left(  t_{1}\right)
...r_{i_{m}}\left(  t_{m}\right)  y_{i_{1}...i_{m}}\right\vert ^{r}%
dt_{1}...dt_{m}\right)  ^{\frac{1}{r}}\\
&  =\lim_{N\rightarrow\infty}\prod_{j=1}^{m}\left(
{\displaystyle\int\limits_{I}}
\left\vert \sum_{i_{j}=1}^{N}r_{i_{j}}\left(  t_{j}\right)  a_{i_{j}%
}\right\vert ^{r}dt_{j}\right)  ^{\frac{1}{r}}\\
&  =\prod_{j=1}^{m}\lim_{N\rightarrow\infty}\left(
{\displaystyle\int\limits_{I}}
\left\vert \sum_{i_{j}=1}^{N}r_{i_{j}}\left(  t_{j}\right)  a_{i_{j}%
}\right\vert ^{r}dt_{j}\right)  ^{\frac{1}{r}}\\
&  =\prod_{j=1}^{m}\left(  \sqrt{2}\left(  \frac{\Gamma\left(  \frac{r+1}%
{2}\right)  }{\sqrt{\pi}}\right)  ^{\frac{1}{r}}\right) \\
&  =\left(  \mathrm{A}_{r}^{-1}\right)  ^{m}.
\end{align*}
Thus, asymptotically we obtain%
\[
K_{m,r}\geq\left(  \mathrm{A}_{r}\right)  ^{m},
\]
for $r\in\left(  p_{0},2\right)  $ and $m\in%
\mathbb{N}
$.
\end{proof}

\section{Application: Optimal constants in the multilinear mixed $\left(
\ell_{\frac{p}{p-1}},\ell_{2}\right)  $-Littlewood inequality}

The objective in this section is to prove that the multiple Khintchine
inequality is equivalent to the multilinear mixed $\left(  \ell_{\frac{p}%
{p-1}},\ell_{2}\right)  $-Littlewood inequality (see \cite{racsam}). As
application of the Section \ref{multi}, we obtain the optimal constants of the
multilinear mixed $\left(  \ell_{\frac{p}{p-1}},\ell_{2}\right)  $-Littlewood
inequality, recovering and ending completely the recent estimates obtained in
\cite{racsam, pell}.

Let us start this section using the Khintchine inequality for multiple sums in
order to prove the following result, also called mixed $\left(  \ell_{\frac
{p}{p-1}},\ell_{2}\right)  $-Littlewood inequality (or multilinear mixed
$\left(  \ell_{\frac{p}{p-1}},\ell_{2}\right)  $-Littlewood inequality) (see
\cite{racsam}). Mimicking the proof of Theorem \ref{ppp} we obtain:

\begin{theorem}
[Multilinear mixed $\left(  \ell_{\frac{p}{p-1}},\ell_{2}\right)  $-Littlewood
inequality]\label{multiracsam} Let $M\geq3$ be a positive integer and $p\geq
2$. There is an optimal constant $C_{(M),p}$ such that
\[
\left(  \sum_{i_{1}=1}^{\infty}\left(  \sum_{i_{2},...,i_{M}=1}^{\infty
}|T(e_{i_{1}},...,e_{i_{M}})|^{2}\right)  ^{\frac{1}{2}\frac{p}{p-1}}\right)
^{\frac{p-1}{p}}\leq C_{(M),p}\Vert T\Vert,
\]
for all continuous $M$-linear forms $T:X_{p}\times X_{\infty}\times\dots\times
X_{\infty}\rightarrow\mathbb{R}$. Moreover, $C_{(M),p}\leq\left(
\mathrm{A}_{p^{\ast}}\right)  ^{M-1}$.
\end{theorem}

\begin{proof}
Let $N$ be a positive integer and $T:X_{p}^{N}\times X_{\infty}^{N}\times
\dots\times X_{\infty}^{N}\rightarrow\mathbb{R}$ be a continuous $M$-linear
form. By the Theorem \ref{multikhin} we know that%

\begin{align*}
&  \left(  \sum_{i_{1}=1}^{N}\left(  \sum_{i_{2},...,i_{M}=1}^{N}|T(e_{i_{1}%
},...,e_{i_{M}})|^{2}\right)  ^{\frac{1}{2}\frac{p}{p-1}}\right)  ^{\frac
{p-1}{p}}\\
&  \quad\leq(\mathrm{A}_{\frac{p}{p-1}})^{M-1}\left(  \sum_{i_{1}=1}^{N}%
\int_{[0,1]^{M-1}}\left\vert \sum_{i_{2},...,i_{M}}^{N}r_{i_{2}}(t_{2})\cdots
r_{i_{M}}(t_{M})T(e_{i_{1}},...,e_{i_{M}})\right\vert ^{\frac{p}{p-1}}%
dt_{2}\cdots dt_{M}\right)  ^{\frac{p-1}{p}}\\
&  \quad=(\mathrm{A}_{\frac{p}{p-1}})^{M-1}\left(  \int_{[0,1]^{M-1}}%
\sum_{i_{1}=1}^{N}\left\vert T\left(  e_{i_{1}},\sum_{i_{2}=1}^{N}r_{i_{2}%
}(t_{2})e_{i_{2}},...,\sum_{i_{M}=1}^{N}r_{i_{M}}(t_{M})e_{i_{M}}\right)
\right\vert ^{\frac{p}{p-1}}dt_{2}\cdots dt_{M}\right)  ^{\frac{p-1}{p}}\\
&  \quad\leq(\mathrm{A}_{\frac{p}{p-1}})^{M-1}\left(  \int_{[0,1]^{M-1}%
}\left\Vert T\left(  \cdot,\sum_{i_{2}=1}^{N}r_{i_{2}}(t_{2})e_{i_{2}%
},...,\sum_{i_{M}=1}^{N}r_{i_{M}}(t_{M})e_{i_{M}}\right)  \right\Vert
^{\frac{p}{p-1}}dt_{2}\cdots dt_{M}\right)  ^{\frac{p-1}{p}}\\
&  \quad\leq(\mathrm{A}_{\frac{p}{p-1}})^{M-1}\sup_{t_{2},...,t_{M}\in
\lbrack0,1]}\left\Vert T\left(  \cdot,\sum_{i_{2}=1}^{N}r_{i_{2}}%
(t_{2})e_{i_{2}},...,\sum_{i_{M}=1}^{N}r_{i_{M}}(t_{M})e_{i_{M}}\right)
\right\Vert \\
&  \quad\leq(\mathrm{A}_{\frac{p}{p-1}})^{M-1}\Vert T\Vert,
\end{align*}
and the proof is done, with estimates $C_{(M),p}\leq\left(  \mathrm{A}%
_{\frac{p}{p-1}}\right)  ^{M-1}$.
\end{proof}

In \cite{racsam}, the authors obtain the estimate $\left(  2^{\frac{1}%
{2}-\frac{1}{p}}\right)  ^{M-1}\leq C_{(M),p}\leq\left(  \mathrm{A}_{\frac
{p}{p-1}}\right)  ^{M-1}$, for all $p\geq2$. Particularly, for all $p\geq
\frac{p_{0}}{p_{0}-1}\approx2.18006$, $C_{(M),p}=\left(  \mathrm{A}_{\frac
{p}{p-1}}\right)  ^{M-1}.$

\medskip It occurs to us to ask about the relation between the Theorem
\ref{multikhin} and Theorem \ref{multiracsam}. We are going to see that, as
occurs in the classical inequalities in the Section \ref{classic}, these two
theorems are also equivalent. Moreover, we will prove that for all $p\geq2$,
$C_{(M),p}=\left(  \mathrm{A}_{\frac{p}{p-1}}\right)  ^{M-1},$ recovering and
completing the estimates of the papers \cite{racsam, pell} for the multilinear case.

We saw that multiple Khintchine inequality implies Theorem \ref{multiracsam}.
On the other hand, proceeding as in Theorem \ref{pppp} and using Lemma
\ref{mink.seq} several times,\ from Theorem \ref{multiracsam} it is not
difficult to prove that in fact we have:

\begin{theorem}
\label{multibom}Let $M\geq3$ be a positive integer and $p\geq2$. There is an
optimal constant $D_{(M),p}$ such that
\[
\left(  \sum_{i_{2},...,i_{M}=1}^{\infty}\left(  \sum_{i_{1}=1}^{\infty
}|T(e_{i_{1}},...,e_{i_{M}})|^{\frac{p}{p-1}}\right)  ^{2\frac{p-1}{p}%
}\right)  ^{\frac{1}{2}}\leq D_{(M),p}\Vert T\Vert,
\]
for all continuous $M$-linear forms $T:X_{p}\times X_{\infty}\times\dots\times
X_{\infty}\rightarrow\mathbb{R}$. Moreover, $D_{(M),p}\leq C_{(M),p}$.
\end{theorem}

Our goal is to prove that this theorem implies the multiple Khintchine
inequality. In order to achieve it, let us recall that for a continuous
$m-$linear form $T:X_{p_{1}}\times\cdots\times X_{p_{m}}\rightarrow
\mathbb{R},$%
\begin{align*}
\left\Vert T\right\Vert  &  =\sup\left\{  \left\vert T\left(  x^{\left(
1\right)  },...,x^{\left(  m\right)  }\right)  \right\vert :\left\Vert
x^{\left(  i\right)  }\right\Vert _{X_{p_{i}}}\leq1;~1\leq i\leq m\right\} \\
&  =\sup\left\{  \left\vert \sum_{i_{1},...,i_{m}}a_{i_{1}\cdots i_{m}%
}x_{i_{1}}^{\left(  1\right)  }...x_{i_{m}}^{\left(  m\right)  }\right\vert
:\left\Vert x^{\left(  i\right)  }\right\Vert _{X_{p_{i}}}\leq1;~1\leq i\leq
m\right\}  ,
\end{align*}
where $p_{1},...,p_{m}\in\lbrack1,\infty]$ and $T\left(  e_{i_{1}%
},...,e_{i_{m}}\right)  =a_{i_{1}\cdots i_{m}}$, for all $i_{j}\in\mathbb{N}.$

\begin{theorem}
Theorem \ref{multibom} implies multiple Khintchine inequality for $p\in
\lbrack1,2]$. Moreover, $\left(  \mathrm{A}_{p}\right)  ^{m-1}\leq
D_{(m),p^{\ast}}$.
\end{theorem}

\begin{proof}
The proof uses the ideas of the proof of the Theorem \ref{ppp}. We are going
to give the details for the case $m=2$. The proof is not different in the
general case, merely more notationally complicated.

Let $p\in\left[  1,2\right]  $, $n\in\mathbb{N}$, and $\left(  y_{jk}\right)
_{j,k=1}^{n}$ an array of scalars. By (\ref{equalm}), for any choice of signs
$\left(  \varepsilon_{j}\right)  _{j=1}^{n},\left(  \lambda_{k}\right)
_{k=1}^{n}\in\{-1,1\}^{n}$ we have%
\[
\left(  \int_{I^{2}}\left\vert \sum_{j,k=1}^{n}y_{jk}r_{j}(t)r_{k}%
(s)\right\vert ^{p}dtds\right)  ^{\frac{1}{p}}=\left(  \int_{I^{2}}\left\vert
\sum_{j,k=1}^{n}y_{jk}r_{j}(t)r_{k}(s)\varepsilon_{j}\lambda_{k}\right\vert
^{p}dtds\right)  ^{\frac{1}{p}}.
\]
Now, by Proposition \ref{multipopa}, solving the integral on the right hand
and rewriting and ordering the indexes we have
\begin{align*}
\left(  \int_{I^{2}}\left\vert \sum_{j,k=1}^{n}y_{jk}r_{j}(t)r_{k}%
(s)\right\vert ^{p}dtds\right)  ^{\frac{1}{p}} &  =\left(  \left(  \frac{1}%
{2}\right)  ^{2n}\sum_{\left(  \alpha,\delta\right)  \in D_{n}\times D_{n}%
}\left\vert \sum_{j,k=1}^{n}y_{jk}\alpha_{j}\delta_{k}\varepsilon_{j}%
\lambda_{k}\right\vert ^{p}\right)  ^{\frac{1}{p}}\\
&  =\left(  \left(  \frac{1}{2}\right)  ^{2n}\sum_{i=1}^{2^{2n}}\left\vert
\sum_{j,k=1}^{n}y_{jk}a_{jk}^{\left(  i\right)  }\varepsilon_{j}\lambda
_{k}\right\vert ^{p}\right)  ^{\frac{1}{p}}%
\end{align*}
where each $a_{jk}^{\left(  i\right)  }$ is $1$ or $-1.$

Hence, if $\left(  x_{i}\right)  _{i=1}^{\infty}\in B_{X_{p^{\ast}}}$ by
H\"{o}lder's inequality it holds
\begin{align}
\left\vert \left(  \frac{1}{2}\right)  ^{2n}\sum_{i=1}^{2^{2n}}\sum
_{j,k=1}^{n}y_{jk}a_{jk}^{\left(  i\right)  }\varepsilon_{j}\lambda_{k}%
x_{i}\right\vert  & \leq\left(  \left(  \frac{1}{2}\right)  ^{2n}\sum
_{i=1}^{2^{2n}}\left\vert \sum_{j,k=1}^{n}y_{jk}a_{jk}^{\left(  i\right)
}\varepsilon_{j}\lambda_{k}\right\vert ^{p}\right)  ^{\frac{1}{p}}\left(
\sum_{i=1}^{2^{2n}}\left\vert x_{i}\right\vert ^{p^{\ast}}\right)  ^{\frac
{1}{p^{\ast}}}\label{final}\\
& =\left(  \int_{I^{2}}\left\vert \sum_{j,k=1}^{n}y_{jk}r_{j}(t)r_{k}%
(s)\right\vert ^{p}dtds\right)  ^{\frac{1}{p}}\left\Vert \left(  x_{i}\right)
_{i=1}^{2^{2n}}\right\Vert _{p^{\ast}}\\
& \leq\left(  \int_{I^{2}}\left\vert \sum_{j,k=1}^{n}y_{jk}r_{j}%
(t)r_{k}(s)\right\vert ^{p}dtds\right)  ^{\frac{1}{p}}.
\end{align}
Define the $3-$linear form $A:X_{p^{\ast}}^{2^{2n}}\times X_{\infty}^{n}\times
X_{\infty}^{n}\rightarrow\mathbb{R}$, such that $A\left(  e_{i},e_{j}%
,e_{k}\right)  =\left(  \frac{1}{2}\right)  ^{\frac{2n}{p}}y_{jk}%
a_{jk}^{\left(  i\right)  }$, for $1\leq j,k\leq n$, $1\leq i\leq2^{2n}$.
Clearly $A$ is bounded, then by Theorem \ref{multibom} ($M=3$) and
\ref{final}:
\begin{align*}
&  \left(  \sum_{j,k=1}^{n}\left\vert y_{jk}\right\vert ^{2}\right)
^{\frac{1}{2}}\\
&  =\left(  \sum_{j,k=1}^{n}\left(  \sum_{i=1}^{2^{2n}}\left\vert \left(
\frac{1}{2}\right)  ^{\frac{2n}{p}}y_{jk}a_{jk}^{\left(  i\right)
}\right\vert ^{p}\right)  ^{\frac{2}{p}}\right)  ^{\frac{1}{2}}\\
&  =\left(  \sum_{j,k=1}^{n}\left(  \sum_{i=1}^{2^{2n}}\left\vert A\left(
e_{i},e_{j},e_{k}\right)  \right\vert ^{p}\right)  ^{\frac{2}{p}}\right)
^{\frac{1}{2}}\\
&  \leq D_{(3),p^{\ast}}\Vert A\Vert\\
&  =D_{(3),p^{\ast}}\sup\left\{  \left\vert \left(  \frac{1}{2}\right)
^{2n}\sum_{i=1}^{2^{2n}}\sum_{j,k=1}^{n}y_{jk}a_{jk}^{\left(  i\right)
}\varepsilon_{j}\lambda_{k}x_{i}\right\vert :\left\Vert x\right\Vert
_{X_{p^{\ast}}}\leq1,~\varepsilon_{j},\lambda_{k}\in\{-1,1\}\right\}  \\
&  \leq D_{(3),p^{\ast}}\left(  \int_{I^{2}}\left\vert \sum_{j,k=1}^{n}%
y_{jk}r_{j}(t)r_{k}(s)\right\vert ^{p}dtds.\right)  ^{\frac{1}{p}}%
\end{align*}
As follows, the multiple Khintchine inequality is proved for $m=2$, and
$\left(  \mathrm{A}_{p}\right)  ^{2}\leq D_{(3),p^{\ast}}$.

The proof of the general case, $m>2$, on the multiple Khintchine inequality
will follow from Theorem \ref{multibom}, considering the case $M=m+1,$ with
estimate
\[
\left(  \mathrm{A}_{p}\right)  ^{m}\leq D_{(m+1),p^{\ast}},
\]

as asserted.
\end{proof}

In this way, note that we have proved that the multiple Khintchine inequality
is equivalent to the multilinear mixed $\left(  \ell_{\frac{p}{p-1}},\ell
_{2}\right)  $-Littlewood inequality and, in general, for all $p\in\left[
1,2\right]  $ and $m\geq2$%
\[
D_{(m+1),p^{\ast}}=C_{(m+1),p^{\ast}}=\left(  \mathrm{A}_{p}\right)  ^{m},
\]
or, equivalently, for all $p\in\left[  2,\infty\right]  $ and $m\geq2$%
\[
D_{(m+1),p}=C_{(m+1),p}=\left(  \mathrm{A}_{\frac{p}{p-1}}\right)  ^{m}.
\]

As announced, this recovers and completes the estimates of the papers
\cite{racsam, pell} for the general case.

\end{document}